\documentclass[a4paper,11pt]{amsart}

\usepackage{amsmath,amssymb,amsthm}
\usepackage{cite}
\allowdisplaybreaks[0]
\numberwithin{equation}{section}
\newtheorem{theorem}{Theorem}[section]

\newcommand{\R}{\mathbb R}
\newcommand{\C}{\mathbb C}
\newcommand{\Mtilde}{\widetilde M}
\DeclareMathOperator{\grad}{grad}
\DeclareMathOperator{\diver}{div}

\title[Hamiltonian stationary Lagrangian surfaces]
{Hamiltonian stationary Lagrangian surfaces in complex space forms
with constant-length gradient of the mean curvature}

\author[]{Toru Sasahara}

\address{Division of Mathematics, Center for Liberal Arts and Sciences,
Hachinohe Institute of Technology, Hachinohe, Aomori 031-8501, Japan}

\email{sasahara@hi-tech.ac.jp}

\date{}

\begin{document}

\begin{abstract}
We study Hamiltonian stationary Lagrangian surfaces in complex
space forms with nowhere-zero mean curvature vector $H$. We prove
that if the Gaussian curvature is constant and $\grad|H|$ has
positive constant length, then both the surface and the ambient
space are flat. Moreover, the immersion is locally congruent to
the known surface obtained as the product of a Cornu spiral and its
reflection.
\end{abstract}


\def\subjclassname{\textup{2020} Mathematics Subject Classification}

\makeatletter
\let\originalsettitle\@settitle
\def\@settitle{\setlength{\topskip}{12pt}\originalsettitle}
\makeatother
\maketitle
\vspace{-10pt}

\section{Introduction}

Let $\Mtilde^n$ be a K\"ahler manifold of complex dimension $n$
with complex structure $J$ and K\"ahler metric $\langle\,,\,\rangle$.
An $n$-dimensional submanifold $M$ of $\Mtilde^n$ is Lagrangian
if $\langle X,JY\rangle=0$ for all tangent vector
fields $X,Y$ on $M$. A normal vector field $\xi$ of a Lagrangian
submanifold $M$ is a Hamiltonian variation vector field if
$\xi=J\grad f$ for some compactly supported smooth function $f$
on $M$, where $\grad$ denotes the gradient with respect to the
induced metric. A Lagrangian submanifold is Hamiltonian stationary
if the first variation of volume vanishes for every Hamiltonian
variation. Oh~\cite{Oh1993} showed that this condition is equivalent to
\begin{equation}\label{eq:hsl}
 \diver(JH)=0,
\end{equation}
where $H$ denotes the mean curvature vector field of $M$.

Let $\Mtilde^n(4\epsilon)$ be a complete and simply connected complex
space form of complex dimension $n$ and constant holomorphic sectional
curvature $4\epsilon$.
%
%
It was proved in~\cite{SasaharaHarmonic} that a Hamiltonian stationary
Lagrangian surface in a complex space form with nonzero constant
mean curvature has a parallel second fundamental form.
The same paper also classifies Hamiltonian stationary Lagrangian
surfaces in complex space forms with nonconstant harmonic mean
curvature and constant Gaussian curvature $K$.
For these surfaces, $K=\epsilon<0$.

In this paper, we assume that the length of $\grad|H|$ is a positive constant
and that the Gaussian curvature is constant.
In contrast to the harmonic case, we obtain $K=\epsilon=0$.
Our main result is the following theorem.

\begin{theorem}\label{thm:classification}
Let $M$ be a Hamiltonian stationary Lagrangian surface in
$\Mtilde^2(4\epsilon)$ with nowhere-zero mean curvature vector $H$.
Suppose that the Gaussian curvature $K$ is constant and that
the length of $\grad|H|$ is a positive constant.
Then $K=\epsilon=0$. Moreover, the immersion is locally congruent
to an open subset of the surface parametrized by
\begin{equation}\label{eq:immersion}
 F_\kappa(x,y)=
 \left(\int_0^x e^{i\kappa t^2}\,dt,
       \int_0^y e^{-i\kappa t^2}\,dt\right)\in\C^2,
\end{equation}
where $(x,y)\in\R^2\setminus\{(0,0)\}$ and
$\kappa=|\grad|H||$.
Conversely, for every $\kappa>0$, the immersion $F_\kappa$ satisfies
all the assumptions of the theorem.
\end{theorem}

The immersion \eqref{eq:immersion} is the product of a Cornu spiral
and its reflection, as described in
\cite[Section~3, Corollary~1]{AnciauxCastro2011}.

Section~\ref{sec:structure} contains the formulas needed for the proof
of Theorem~\ref{thm:classification}, which is given in
Section~\ref{sec:proof}. In Section~\ref{sec:chen-comparison}, we establish
an explicit congruence between $F_\kappa$ and the $c=0$ extension
of the second family of Chen, Garay, and Zhou~\cite{CGZ2009}.

\section{Preliminaries}
\label{sec:structure}

\subsection{Fundamental equations}

Let $M$ be a Lagrangian submanifold of $\Mtilde^n(4\epsilon)$.
Let $\nabla$ and $\widetilde\nabla$ denote the Levi--Civita
connections on $M$ and $\Mtilde^n(4\epsilon)$, respectively.
The Gauss and Weingarten formulas are
\begin{equation}\label{eq:gauss-weingarten}
 \widetilde\nabla_XY=\nabla_XY+h(X,Y),\qquad
 \widetilde\nabla_X\xi=-S_\xi X+\nabla^\perp_X\xi
\end{equation}
for tangent vector fields $X,Y$ and a normal vector field $\xi$.
Here $h$ is the second fundamental form, $S_\xi$ is the shape
operator in the direction $\xi$, and $\nabla^\perp$ is the normal
connection.
The mean curvature vector field is given by
$H=(1/n)\operatorname{tr}_g h$, where $g$ is the induced metric.
In this paper, we call the function $|H|$ the mean curvature.
All gradients and divergences are taken with respect to $g$.
We use the sign convention $\Delta=-\diver\grad$ for the Laplacian.

The Lagrangian condition implies
\begin{equation}\label{eq:lagrangian-connection}
 \nabla^\perp_X(JY)=J(\nabla_XY).
\end{equation}
Moreover, the cubic form
\begin{equation}\label{eq:lagrangian-cubic}
 \mathcal C(X,Y,Z)=\langle h(X,Y),JZ\rangle
\end{equation}
is symmetric in all three arguments
(see \cite[Section~2]{DillenLiVranckenWang2012}).

We define the curvature tensor by
\begin{equation}\label{eq:curvature-conventions}
 R(X,Y)Z=\nabla_X\nabla_YZ-\nabla_Y\nabla_XZ-\nabla_{[X,Y]}Z.
\end{equation}
The Gauss and Codazzi equations are
\begin{equation}\label{eq:gauss-equation}
\begin{aligned}
 \langle R(X,Y)Z,W\rangle
 &=\langle h(Y,Z),h(X,W)\rangle-\langle h(X,Z),h(Y,W)\rangle\\
 &\quad+\epsilon\bigl(\langle Y,Z\rangle\langle X,W\rangle
                    -\langle X,Z\rangle\langle Y,W\rangle\bigr),
\end{aligned}
\end{equation}
\begin{equation}\label{eq:codazzi-equation}
 (\bar\nabla_Xh)(Y,Z)=(\bar\nabla_Yh)(X,Z),
\end{equation}
where $X,Y,Z,W$ are tangent vector fields and $\bar\nabla h$ is
defined by
\begin{equation}\label{eq:covariant-h}
 (\bar\nabla_Xh)(Y,Z)
 =\nabla^\perp_Xh(Y,Z)-h(\nabla_XY,Z)-h(Y,\nabla_XZ).
\end{equation}
From now on, let $n=2$. For an orthonormal tangent frame
$e_1,e_2$, the Gaussian curvature is given by
$K=\langle R(e_1,e_2)e_2,e_1\rangle$.

\subsection{The mean curvature one-form}

Assume that $M$ is Hamiltonian stationary and that $H$ is nowhere
zero. Define
\begin{equation}\label{eq:mean-trace}
 a=|H|>0,\qquad T=JH,\qquad
 \alpha=T^\flat=g(T,\cdot).
\end{equation}
By definition,
$\alpha(Y)=-\langle H,JY\rangle
=-\tfrac12\sum_{i=1}^2\mathcal C(e_i,e_i,Y)$.
By \eqref{eq:lagrangian-connection}, the Codazzi equation gives
\begin{equation*}
 (\nabla_X\mathcal C)(Y,Z,W)
 =(\nabla_Y\mathcal C)(X,Z,W).
\end{equation*}
Since $\mathcal C$ is symmetric, it follows that
$\nabla\mathcal C$ is symmetric in all four arguments.
Taking a trace, we see that $\nabla\alpha$ is symmetric.
Hence $\alpha$ is closed. By Hamiltonian stationarity, we have
\begin{equation}\label{eq:mean-form-harmonic}
 d\alpha=0,\qquad \diver T=0.
\end{equation}

\section{Proof of the main theorem}
\label{sec:proof}

\begin{proof}
By assumption and \eqref{eq:mean-trace},
$|\grad a|=\kappa>0$, where $\kappa$ is constant.

\smallskip\noindent
\textbf{Step 1. $K=0$.}
By \eqref{eq:mean-form-harmonic}, there is a locally defined smooth
function $u$ satisfying
\begin{equation}\label{eq:harmonic-potential}
 T=\grad u,\qquad \Delta u=0,\qquad |\grad u|=a>0.
\end{equation}
With our sign convention, a harmonic function $v$ on a Riemannian
surface satisfies $\Delta\log|\grad v|=-K$ wherever $\grad v\ne0$
(see \cite[Lemma~2.1]{AdamowiczVeronelli2022}).
Applying this formula to $u$, we obtain
\begin{equation}\label{eq:intrinsic-log-a}
 \Delta\log a=-K.
\end{equation}
For $\Delta=-\diver\grad$, the chain rule takes the form
\begin{equation*}
 \Delta\log a=\frac{\Delta a}{a}+\frac{|\grad a|^2}{a^2}.
\end{equation*}
Combining this with \eqref{eq:intrinsic-log-a} and
$|\grad a|=\kappa$, we obtain
\begin{equation}\label{eq:laplacian-a}
 \Delta a=-Ka-\frac{\kappa^2}{a}.
\end{equation}

The Hessian of a smooth real-valued function $v$ is the symmetric
covariant $2$-tensor defined by
\begin{equation*}
 (\nabla^2v)(X,Y)=\langle\nabla_X\grad v,Y\rangle
 =X(Yv)-(\nabla_XY)v.
\end{equation*}
Since $|\grad a|^2=\kappa^2$, we have
\begin{equation*}
 0=X(|\grad a|^2)=2(\nabla^2a)(X,\grad a)
\end{equation*}
for every tangent vector field $X$. Thus the nonzero vector field
$\grad a$ lies in the kernel of $\nabla^2a$.
Since $M$ is two-dimensional and
$\operatorname{tr}_g\nabla^2a=-\Delta a$, the eigenvalues of
$\nabla^2a$ with respect to $g$ are $0$ and $-\Delta a$.
The squared norm is therefore
\begin{equation}\label{eq:hessian-gradient-frame}
 |\nabla^2a|^2=(\Delta a)^2.
\end{equation}
Since $K$ and $\kappa$ are constant, differentiating 
\eqref{eq:laplacian-a} shows
\begin{equation}\label{eq:gradient-laplacian-a}
 \grad\Delta a=\left(-K+\frac{\kappa^2}{a^2}\right)\grad a.
\end{equation}
For a smooth real-valued function $v$ on a Riemannian surface, the Bochner
formula with our sign convention is
\begin{equation}\label{eq:bochner-formula}
 \tfrac12\Delta|\grad v|^2
 =-|\nabla^2v|^2+\langle\grad v,\grad\Delta v\rangle-K|\grad v|^2.
\end{equation}
Applying \eqref{eq:bochner-formula} to $a$ and using
\eqref{eq:laplacian-a}--\eqref{eq:gradient-laplacian-a}, we obtain
\begin{equation}\label{eq:bochner-compatibility}
\begin{aligned}
 0&=-\left(Ka+\frac{\kappa^2}{a}\right)^2
     +\left(-K+\frac{\kappa^2}{a^2}\right)\kappa^2-K\kappa^2\\
  &=-K(Ka^2+4\kappa^2).
\end{aligned}
\end{equation}
If $K\ne0$, equation~\eqref{eq:bochner-compatibility} gives
$a^2=-4\kappa^2/K$, which is constant.
This contradicts $\grad(a^2)=2a\grad a\ne0$.
Hence $K=0$.

\smallskip\noindent
\textbf{Step 2. A local normal form for $JH$.}
Since $K=0$, choose local coordinates $z=x+iy$ with
$g=dx^2+dy^2$. Write $T=P\partial_x+Q\partial_y$.
In terms of $\alpha=P\,dx+Q\,dy$, equations
\eqref{eq:mean-form-harmonic} read
\begin{equation}\label{eq:trace-cr}
 P_y=Q_x,\qquad P_x=-Q_y.
\end{equation}
Thus $f=P-iQ$ is holomorphic and $a=|f|>0$.
Differentiating $a^2=P^2+Q^2$ and using \eqref{eq:trace-cr}, we obtain
\begin{equation*}
 a_x=\frac{PP_x+QQ_x}{a},\qquad
 a_y=\frac{PQ_x-QP_x}{a}.
\end{equation*}
Consequently,
\begin{equation*}
 a_x^2+a_y^2
 =\frac{(P^2+Q^2)(P_x^2+Q_x^2)}{a^2}
 =P_x^2+Q_x^2.
\end{equation*}
Using $g=dx^2+dy^2$ and $f'=P_x-iQ_x$, we have
\begin{equation}\label{eq:holomorphic-gradient}
 \kappa^2=|\grad a|^2=|f'|^2.
\end{equation}
The holomorphic function $f'$ has constant modulus and is therefore
constant. It follows that
\begin{equation}\label{eq:affine-trace}
 f(z)=\gamma z+\delta,\qquad |\gamma|=\kappa>0,
\end{equation}
where $\gamma,\delta\in\C$ are constants.

Since $\alpha=\operatorname{Re}(f\,dz)$, the coordinate change
$z=z_0+e^{i\beta}\zeta$ transforms \eqref{eq:affine-trace} into
\begin{equation}\label{eq:trace-coordinate-transform}
\begin{aligned}
 \alpha&=\operatorname{Re}\bigl(\widetilde f(\zeta)\,d\zeta\bigr),\\
 \widetilde f(\zeta)&=e^{i\beta}f(z_0+e^{i\beta}\zeta)
 =\gamma e^{2i\beta}\zeta+e^{i\beta}(\gamma z_0+\delta).
\end{aligned}
\end{equation}
Write $\gamma=\kappa e^{i\vartheta}$ and choose
$z_0=-\delta/\gamma$ and $\beta=(\pi-\vartheta)/2$.
Equation~\eqref{eq:trace-coordinate-transform} then gives
$\widetilde f(\zeta)=-\kappa\zeta$. Writing $\zeta=x+iy$, we have
$\alpha=\operatorname{Re}(-\kappa\zeta\,d\zeta)
=-\kappa x\,dx+\kappa y\,dy$.
Since $\alpha=g(T,\cdot)$ and $a=|T|$, we obtain
\begin{equation}\label{eq:normalized-trace}
 T=-\kappa x\partial_x+\kappa y\partial_y,\qquad
 a=\kappa\sqrt{x^2+y^2}>0.
\end{equation}

\smallskip\noindent
\textbf{Step 3. $\epsilon=0$ and determination of $h$.}
By the symmetry of the cubic form \eqref{eq:lagrangian-cubic},
the second fundamental form can be written as
\begin{equation}\label{eq:flat-frame-form}
\begin{aligned}
 h(\partial_x,\partial_x)&=A J\partial_x+B J\partial_y,\\
 h(\partial_x,\partial_y)&=B J\partial_x+C J\partial_y,\\
 h(\partial_y,\partial_y)&=C J\partial_x+D J\partial_y.
\end{aligned}
\end{equation}
Here $A,B,C,D$ are smooth real-valued functions. Taking the trace and
using $H=-JT$ and \eqref{eq:normalized-trace}, we obtain
\begin{equation}\label{eq:trace}
 A+C=2\kappa x,\qquad B+D=-2\kappa y.
\end{equation}
The frame $\partial_x,\partial_y$ is parallel.
By \eqref{eq:lagrangian-connection}, the normal frame
$J\partial_x,J\partial_y$ is also parallel with respect to the normal
connection. Thus, by \eqref{eq:covariant-h}, the Codazzi
equation~\eqref{eq:codazzi-equation} reduces to
\begin{equation}\label{eq:parallel-codazzi}
 A_y=B_x,\qquad B_y=C_x,\qquad C_y=D_x.
\end{equation}
Substituting $A=2\kappa x-C$ and $D=-2\kappa y-B$ from
\eqref{eq:trace} into \eqref{eq:parallel-codazzi} gives
\begin{equation}\label{eq:codazzi-cr}
 C_x=B_y,\qquad C_y=-B_x.
\end{equation}
Thus $\Delta B=\Delta C=0$.

Substituting \eqref{eq:flat-frame-form} into
\eqref{eq:gauss-equation} gives $0=\epsilon+A C+B D-B^2-C^2$.
Eliminating $A$ and $D$ by \eqref{eq:trace}, we obtain
\begin{equation}\label{eq:reduced-gauss}
 \epsilon=2(B^2+C^2)-2\kappa(x C-y B).
\end{equation}
The harmonicity of $B$ and $C$, together with \eqref{eq:codazzi-cr},
implies
\begin{equation}\label{eq:codazzi-laplacians}
\begin{aligned}
 \Delta(B^2+C^2)
 &=-4(B_x^2+C_x^2),\\
 \Delta(x C-y B)&=0.
\end{aligned}
\end{equation}
Since $\epsilon$ is constant, applying $\Delta$ to
\eqref{eq:reduced-gauss} gives
\begin{equation}\label{eq:ambient-flatness}
 0=\Delta\epsilon=-8(B_x^2+C_x^2).
\end{equation}
Equations \eqref{eq:ambient-flatness} and \eqref{eq:codazzi-cr}
imply that $B,C$ are constant. Differentiating
\eqref{eq:reduced-gauss} with respect to $x$ and $y$ then gives
$C=0$ and $B=0$, respectively. Hence $\epsilon=0$.
Substituting into \eqref{eq:trace}, we find $A=2\kappa x$ and
$D=-2\kappa y$. Thus we have 
\begin{equation}\label{eq:fundamental-forms}
\begin{aligned}
 h(\partial_x,\partial_x)&=2\kappa xJ\partial_x,\\
 h(\partial_x,\partial_y)&=0,\\
 h(\partial_y,\partial_y)&=-2\kappa yJ\partial_y.
\end{aligned}
\end{equation}

\smallskip\noindent
\textbf{Step 4. Integration and converse.}
By Step~3, the ambient space is $\C^2$. Let $F$ denote the given
immersion in the coordinates $(x,y)$.
The Gauss formula and \eqref{eq:fundamental-forms} give
\begin{equation}\label{eq:immersion-system}
 F_{xx}=2i\kappa xF_x,\qquad F_{xy}=0,\qquad
 F_{yy}=-2i\kappa yF_y.
\end{equation}
Integrating \eqref{eq:immersion-system}, we obtain
\begin{equation}\label{eq:integrated-tangent-fields}
 F_x=e^{i\kappa x^2}V_1,\qquad
 F_y=e^{-i\kappa y^2}V_2,
\end{equation}
where $V_1,V_2\in\C^2$ are constant vectors.
A second integration yields
\begin{equation}\label{eq:integrated-immersion}
 F(x,y)=V_0+V_1\int_0^x e^{i\kappa t^2}\,dt
              +V_2\int_0^y e^{-i\kappa t^2}\,dt,
 \qquad V_0\in\C^2.
\end{equation}
The induced metric and the Lagrangian condition imply that
$V_1,V_2$ are orthonormal with respect to the Hermitian inner
product. Thus the matrix $W=(V_1\ V_2)$ is unitary.
The holomorphic Euclidean isometry $Z\mapsto W^*(Z-V_0)$ transforms
\eqref{eq:integrated-immersion} into \eqref{eq:immersion}.
Here $W^*$ denotes the conjugate transpose of $W$.
The origin is excluded by \eqref{eq:normalized-trace}.

For the converse, fix $\kappa>0$ and let $F=F_\kappa$ be the map
in \eqref{eq:immersion}. Differentiation gives
\begin{equation}\label{eq:model-derivatives}
 F_x=(e^{i\kappa x^2},0),\qquad
 F_y=(0,e^{-i\kappa y^2}).
\end{equation}
These tangent vectors are orthonormal with respect to the Hermitian
inner product. Hence $F_\kappa$
is a Lagrangian immersion with $g=dx^2+dy^2$ and $K=0$.
Differentiating \eqref{eq:model-derivatives} again gives
\eqref{eq:immersion-system}, and hence \eqref{eq:fundamental-forms}.
The mean curvature vector and its length are given by
\begin{equation}\label{eq:model-mean-curvature}
\begin{aligned}
 H&=\tfrac12(F_{xx}+F_{yy})
     =\kappa(xJF_x-yJF_y),\\
 |H|&=\kappa\sqrt{x^2+y^2}.
\end{aligned}
\end{equation}
Thus $H\ne0$ away from the origin. Moreover,
\eqref{eq:model-mean-curvature} gives
$T=JH=-\kappa x\partial_x+\kappa y\partial_y$ in the induced
coordinates. Thus $\diver(JH)=-\kappa+\kappa=0$, and
$F_\kappa$ is Hamiltonian stationary.

Finally, \eqref{eq:model-mean-curvature} and the Euclidean metric give
\begin{equation}\label{eq:model-stationarity-gradient}
 \grad|H|=\frac{\kappa}{\sqrt{x^2+y^2}}
                 (x\partial_x+y\partial_y),\qquad
 |\grad|H||=\kappa>0.
\end{equation}
This proves the converse.
\end{proof}

\section{Comparison with the Chen--Garay--Zhou family}
\label{sec:chen-comparison}

We compare the model $F_\kappa$ with the second family $L_2$
in \cite[Theorem~9.1]{CGZ2009}. The proof of
Theorem~\ref{thm:classification} does not depend on this comparison.

We write $B_0>0$ for the scale parameter denoted by $a$ in
\cite{CGZ2009}, to distinguish it from the function $a=|H|$.
Although the theorem assumes $c\ne0$, setting $c=0$
in \cite[(9.16)]{CGZ2009} gives a map
$L_{2,0}$, which we call the \emph{$c=0$ extension}.
We show below that this map satisfies the
assumptions of Theorem~\ref{thm:classification}.\footnote{In \cite[p.~2658, Theorem~9.1]{CGZ2009},
the first two terms in parentheses in the second component of
$L_2(r,\theta)$ are
$ir^2J_{(1+ic)/2}(r^2)T_c^+(r,\theta)
-r^2J_{(ic-1)/2}(r^2)T_c^-(r,\theta)$
with a factor containing $r^2$ outside the parentheses.
These terms should be
$iJ_{(1+ic)/2}(r^2)T_c^+(r,\theta)
-J_{(ic-1)/2}(r^2)T_c^-(r,\theta)$.
Indeed, the coefficient of $-c_2$ in \cite[p.~2657, (9.16)]{CGZ2009}
contains only one factor of $r^2$ in each term. The same correction
applies to the display immediately preceding Theorem~9.1.
Throughout this section, we use the formula in~\cite[(9.16)]{CGZ2009}.}

\subsection{Specialization at $c=0$}

Let $J_\nu$ denote the Bessel function of the first kind of
order $\nu$. For $c=0$, we use the following formulas for $s>0$:
\begin{equation}\label{eq:bessel-half-orders}
 J_{-1/2}(s)=\sqrt{\frac{2}{\pi s}}\cos s,\qquad
 J_{1/2}(s)=\sqrt{\frac{2}{\pi s}}\sin s.
\end{equation}

In the polar coordinates $(r,\theta)$ of \cite[(9.3)]{CGZ2009},
with $r>0$ and $0<\theta<\pi/2$, set
\begin{equation}\label{eq:chen-integrals}
 \mathcal I(s)=\int_0^s e^{2it^2}\,dt,\qquad
 T_0^\pm(r,\theta)=\int_0^\theta
       e^{\pm it+ir^2\cos(2t)}\,dt.
\end{equation}
Here $s\in\R$, and $T_0^\pm$ are the $c=0$ specializations of
the functions in \cite[(9.9)]{CGZ2009}. The formula in
\cite[(9.16)]{CGZ2009} becomes
\begin{equation}\label{eq:chen-zero-bessel}
 L_{2,0}=c_1\Phi_1-c_2\Phi_2,
\end{equation}
where $c_1,c_2\in\C^2$ are constant vectors and
\begin{equation*}
\begin{aligned}
 \Phi_1
 &=r^2\bigl(iJ_{-1/2}(r^2)T_0^+
                +J_{1/2}(r^2)T_0^-\bigr)\\
 &\quad+\int_0^r t e^{it^2}
           \bigl(J_{-1/2}(t^2)+iJ_{1/2}(t^2)\bigr)\,dt,\\
 \Phi_2
 &=r^2\bigl(iJ_{1/2}(r^2)T_0^+
                -J_{-1/2}(r^2)T_0^-\bigr)\\
 &\quad+\int_0^r t e^{it^2}
           \bigl(J_{1/2}(t^2)-iJ_{-1/2}(t^2)\bigr)\,dt.
\end{aligned}
\end{equation*}
Here and below, we write $T_0^\pm$ for $T_0^\pm(r,\theta)$.
By \eqref{eq:bessel-half-orders}, the radial integrals in
$\Phi_1$ and $\Phi_2$ are
$\sqrt{2/\pi}\,\mathcal I(r)$ and
$-i\sqrt{2/\pi}\,\mathcal I(r)$, respectively.

At $c=0$, the normalization in \cite[(9.17)]{CGZ2009} requires
$c_1$ and $c_2$ to be orthogonal with respect to the Hermitian inner
product, each with squared norm $2\pi B_0^2$.
We therefore choose
\begin{equation*}
 c_1=\sqrt{2\pi}B_0(1,0),\qquad
 c_2=-\sqrt{2\pi}B_0(0,1).
\end{equation*}
Substituting these vectors into \eqref{eq:chen-zero-bessel} and
using \eqref{eq:bessel-half-orders}, we obtain
\begin{equation}\label{eq:chen-zero-specialization}
 L_{2,0}(r,\theta)=2B_0
 \begin{pmatrix}
 r\bigl(i\cos(r^2)T_0^++\sin(r^2)T_0^-\bigr)+\mathcal I(r)\\
 r\bigl(i\sin(r^2)T_0^+-\cos(r^2)T_0^-\bigr)-i\mathcal I(r)
 \end{pmatrix}.
\end{equation}
\subsection{Elimination of $T_0^\pm$}

Differentiating \eqref{eq:chen-integrals} with respect to
$\theta$ gives
\begin{equation}\label{eq:chen-angular-derivatives}
\begin{aligned}
 &\partial_\theta
   \bigl[r\bigl(i\cos(r^2)T_0^++\sin(r^2)T_0^-\bigr)\bigr]\\
 &\qquad=-r\sin\theta\,e^{2ir^2\cos^2\theta}
            +ir\cos\theta\,e^{-2ir^2\sin^2\theta},\\
 &\partial_\theta
   \bigl[r\bigl(i\sin(r^2)T_0^+-\cos(r^2)T_0^-\bigr)\bigr]\\
 &\qquad=ir\sin\theta\,e^{2ir^2\cos^2\theta}
            -r\cos\theta\,e^{-2ir^2\sin^2\theta}.
\end{aligned}
\end{equation}
Integrating \eqref{eq:chen-angular-derivatives} from $0$ to
$\theta$, we obtain
\begin{equation*}
\begin{aligned}
 r\bigl(i\cos(r^2)T_0^++\sin(r^2)T_0^-\bigr)
 &=\mathcal I(r\cos\theta)
       +i\overline{\mathcal I(r\sin\theta)}-\mathcal I(r),\\
 r\bigl(i\sin(r^2)T_0^+-\cos(r^2)T_0^-\bigr)
 &=-i\mathcal I(r\cos\theta)
       -\overline{\mathcal I(r\sin\theta)}+i\mathcal I(r).
\end{aligned}
\end{equation*}
Here the overline denotes complex conjugation.
Substituting these expressions into \eqref{eq:chen-zero-specialization},
we arrive at
\begin{equation}\label{eq:chen-zero-map}
 L_{2,0}(r,\theta)=2B_0
 \begin{pmatrix}
 \mathcal I(r\cos\theta)+i\overline{\mathcal I(r\sin\theta)}\\
 -i\mathcal I(r\cos\theta)-\overline{\mathcal I(r\sin\theta)}
 \end{pmatrix}.
\end{equation}
\subsection{Coordinate change and holomorphic congruence}

At $c=0$, the metric in \cite[(9.4)]{CGZ2009} is
$g=8B_0^2(dr^2+r^2d\theta^2)$. Set
\begin{equation}\label{eq:chen-model-coordinates}
 x=2\sqrt2 B_0r\cos\theta,\qquad
 y=2\sqrt2 B_0r\sin\theta,\qquad
 \kappa=\frac1{4B_0^2}.
\end{equation}
Then $g=dx^2+dy^2$, and the substitution $t=2\sqrt2 B_0s$
in \eqref{eq:immersion} gives
\begin{equation*}
 F_\kappa(x,y)=2\sqrt2 B_0
 \begin{pmatrix}
 \mathcal I(r\cos\theta)\\
 \overline{\mathcal I(r\sin\theta)}
 \end{pmatrix}.
\end{equation*}
Comparing this  with \eqref{eq:chen-zero-map}, we obtain
\begin{equation}\label{eq:chen-unitary-congruence}
 L_{2,0}(r,\theta)=U F_\kappa(x,y),\qquad
 U=\frac1{\sqrt2}\begin{pmatrix}1&i\\-i&-1\end{pmatrix},\qquad
 U^*U=I_2.
\end{equation}
Thus $U$ defines a holomorphic isometry of $\C^2$.

Although the original chart covers only the first quadrant $x>0$, $y>0$,
the right-hand side of \eqref{eq:chen-zero-specialization} defines
a smooth extension $\widehat L_{2,0}$ to $r>0$, $\theta\in\R$.
The equalities used to obtain \eqref{eq:chen-zero-map} hold
for all real $\theta$, since both sides of each equality have the same
derivative with respect to $\theta$ and agree at $\theta=0$.
Thus \eqref{eq:chen-unitary-congruence} holds with $L_{2,0}$ replaced
by $\widehat L_{2,0}$ on this extended domain.
The polar coordinate change \eqref{eq:chen-model-coordinates} covers
$\R^2\setminus\{(0,0)\}$ and has Jacobian determinant $8B_0^2r>0$.
Hence the congruence holds locally near every point of this punctured plane,
including points on the coordinate axes. The origin is excluded because
$H=0$ there by \eqref{eq:model-mean-curvature}.

This proves that the $c=0$ extension is locally congruent to $F_\kappa$,
with $B_0=(4\kappa)^{-1/2}$.

\end{document}